\documentclass[11pt]{amsart}
\usepackage{amssymb,amsmath,amsthm} 
\usepackage[english]{babel}
\usepackage[utf8x]{inputenc}
\usepackage[T1]{fontenc}
\usepackage{enumerate}
\usepackage{mathtools}
\usepackage{graphicx}
\usepackage{subfigure} 
\usepackage{float}
\graphicspath{{Img/}} 
\usepackage[export]{adjustbox}
\usepackage{tikz}
\usetikzlibrary{matrix}
\usepackage{multirow}
\usepackage{bbm}
\usepackage{dsfont}

\newtheorem{theorem}{Theorem}[section]
\newtheorem{proposition}[theorem]{Proposition}
\newtheorem{cor}[theorem]{Corollary}
\newtheorem{lemma}[theorem]{Lemma}
\newtheorem{que}{Question}
\newtheorem*{que*}{Question}

\theoremstyle{definition}
\newtheorem{definition}[theorem]{Definition}

\theoremstyle{remark}

\newtheorem{claim}[theorem]{Claim}

\newcommand{\R}{\mathbb{R}}

\newcommand{\Q}{\mathbb{Q}}

\newcommand{\Aa}{\mathcal{A}}
\newcommand{\Ff}{\mathcal{F}}

\newcommand{\Vv}{\mathcal{V}}
\newcommand{\Uu}{\mathcal{U}}

\usepackage{amsfonts}
\usepackage{amsthm}
\usepackage{cancel}
\usepackage{amsmath}
\usepackage{mdframed}
\usepackage{hyperref}

\usepackage{array,booktabs}
\usepackage{bm}
\usepackage{tcolorbox}
\usepackage{enumitem}
\usepackage{amssymb}
\usepackage{enumitem}
\usepackage{tikz}
\DeclareMathOperator{\interior}{int}

\newtheoremstyle{less}      % Name of the style
  {1pt}                           % Space above
  {1pt}                           % Space below
  {\itshape}                     % Body font (italic)
  {1em}                              % Indent amount
  {\normalfont}                   % Theorem head font (not bold)
  {.}                             % Punctuation after theorem head
  { }                             % Space after theorem head
  {}

\theoremstyle{plain}

\newtheorem{observation}{Observation}

\newcommand{\N}{\mathbb{N}}

\newcounter{alphsubsection}[section]

\newcounter{Alphsubsection}[section] % Reset at each new section

\newtcolorbox{mybox}[3][]
{
  colframe = #2!25,
  colback  = #2!10,
  coltitle = #2!20!black,  
  title    = {#3},
  #1,
}

\newtcolorbox{simplebox}{
  colframe=black,    % black frame
  colback=white,     % white background
  boxrule=0.5pt,     % thin border
  arc=0mm,           % rounded corners
  left=1mm,          % padding left
  right=1mm,         % padding right
  top=1mm,           % padding top
  bottom=1mm         % padding bottom
}

\title{Normality in the square of the Sorgenfrey line}

\author{Paul Szeptycki}
\address{Department of Mathematics and Statistics, York University, 4700 Keele St, Toronto, ON, Canada, M3J 1P3.}
\email{szeptyck@yorku.ca}

\author{Hongwei Wen}
\address{Department of Mathematics and Statistics, York University, 4700 Keele St, Toronto, ON, Canada, M3J 1P3.}
\email{wenhongw@yorku.ca}
\date{}



\thanks{The first author acknowledges support from NSERC Grant RGPIN-2025-06324.}
\begin{document}

%%%%%%%%%%%%%%%%%%%%%%%%%%%%%%%%%%%%%%%%%%%%%%%%%%%%%%%%%%%%
%%%%%%%%%%%%%%%%%%%%%%%%%%%%%%%%%%%%%%%%%%%%%%%%%%%%%%%%%%%%
% This a placeholder for the TOPLOGY PROCEEDINGS logo %%%%%%
\noindent                                             %%%%%%
\begin{picture}(150,36)                               %%%%%%
\put(5,20){\tiny{Submitted to}}                       %%%%%%
\put(5,7){\textbf{Topology Proceedings}}              %%%%%%
\put(0,0){\framebox(140,34){}}                        %%%%%%
\put(2,2){\framebox(136,30){}}                        %%%%%%
\end{picture}                                         %%%%%%
%%%%%%%%%%%%%%%%%%%%%%%%%%%%%%%%%%%%%%%%%%%%%%%%%%%%%%%%%%%%
%%%%%%%%%%%%%%%%%%%%%%%%%%%%%%%%%%%%%%%%%%%%%%%%%%%%%%%%%%%%
\vspace{0.5in}

\begin{abstract} We consider sets of reals X endowed with the Sorgenfrey lower limit topology denoted $X[\leq]$. Przymusi\'nski proved that if X is a Q-set then $(X[\leq])^2$
is normal. 
Todor\v cevi\'c proved if X is an entangled set, then all finite power of $X[\leq]$ are hereditarily Lindel\"{o}f, hence $(X[\leq])^2$ is normal. 
We construct from CH a subset X of reals, such that X is not a $\lambda$-set (hence not a Q-set), X is not 2-entangled, but still $(X[\leq])^2$
is normal. We will also prove the natural analogue of Przymusi\'nski’s theorem for $\lambda$-sets, namely, if X is a $\lambda$-set then $(X[\leq])^2$ is pseudo-normal.

\end{abstract}

\renewcommand{\bf}{\bfseries}
\renewcommand{\sc}{\scshape}
%insert defs/styles
\vspace{0.5in}
\maketitle

\section{Introduction}
A $Q$-set is an uncountable subset of the reals with the property that every subset is $G_\delta$ (equivalently, $F_\sigma$) in the relative Euclidean subspace topology. The notion was first introduced in \cite{hausdorff1933} and much later $Q$-sets were used to produce consistent counterexamples to the Normal Moore Space Conjecture \cite{Bing1951}. For example, if $X\subseteq 2^\omega$, the set of branches $b_x=\{x\upharpoonright n:n<\omega\}$ in $2^{<\omega}$ is an almost disjoint family, and the corresponding $\Psi$-space on $2^{<\omega}$ is normal if and only if $X$ is a $Q$-set. Similarly, if $X$ is a subset of the $x$-axis in the upper half plane of ${\mathbb R}^2$, then the Moore-Niemyitzki Plane topology (a.k.a. the bubble space) is normal if and only if $X$ is a Q-set. For more on these examples see \cite{T}. On the other hand Prymusi\'nski \cite{T.Przymusinski1973} proved that if $X$ is a Q-set endowed with the Sorgenfrey Topology then $X^2$ is normal. 

Recall, the Sorgenfrey line is the real line endowed with the topology, also known as the lower limit topology, generated by using the half-open intervals $[a,b)$ as a basis. For a set of reals $X$ we will denote $X[\leq]$ as the resulting topology taking $X$ as a subspace of the Sorgenfrey line.

We first observe that the converse to Przymusi\'nski's theorem does not hold: Todor\v cevi\'c observed that if $X$ is an entangled set \cite{Todorcevic1985} of reals and is endowed with the Sorgenfrey topology, then every finite power is both hereditarily Lindel\"of and hereditarily separable \cite{Todorcevic1989}.  Recall that CH implies both the existence of entangled sets of reals and the non-existence of $Q$-sets. So $(X[\leq])^{2}$ may be normal even if $X$ is not a $Q$-set. 

% It seems that the question whether this also characterizes being a $Q$-set has not been previously considered.

There is the weaker notion of a set of reals $X$ being a $\lambda$-set \cite{Miller1984}. $X$ is a $\lambda$-set if every countable subset is relatively $G_\delta$ in the Euclidean topology. While the existence of a $Q$-set of reals is independent of ZFC, there are $\lambda$-sets in ZFC. Analogous to the $Q$-set results mentioned above, the $\Psi$-space over the set of branches of a subset $X\subseteq 2^\omega$ is pseudo-normal (meaning that every pair of closed sets, one of which is countable, can be separated by disjoint open sets) if and only if $X$ is a $\lambda$-set and also $X$ is a $\lambda$-set if and only if the Moore-Niemytzki plane over $X$ is pseudo-normal (see \cite{vD}).

 % We show that the expected analogous result for the Sorgenfrey topology on a $\lambda$-set holds.  Namely, if $X$ is a $\lambda$-set of reals, then the square of $X$ endowed with the Sorgenfrey topology is pseudo-normal. We also show that assuming CH there is an uncountable subset of the Sorgenfrey line whose square is normal. Since no set of reals of size continuum can be a $Q$-set, this shows that the converse to Przymusi\'nski's theorem does not hold. Moreover, the set can be constructed so as to not even be a $\lambda$-set.  
 We provide some partial converses to Przymusi\'nski's theorem underscoring necessary conditions for normality of $(X[\leq])^2$.
  We show that the expected analogous result for the Sorgenfrey topology on a $\lambda$-set holds. Namely, if $X$ is a $\lambda$-set of reals, then $(X[\leq])^2$ is pseudo-normal. We also construct, assuming CH, a set of reals for which $X$ is concentrated on a countable dense subset (so not a $\lambda$-set) and so that $X$ is also not 2-entangled, but still $X[\leq]^2$ is normal.

%We employ the ideas in Przymusi\'nski's original proof of normality of the square of a Q-set in our proofs of pseudo-normality of the $\lambda$-set and the CH construction. 

We will use the following terminology and notation throughout. 
\begin{definition}
    Let $X$ be a topological space. Let $A,B$ be subsets in $X$, $A$ has the \textbf{covering property with respect to $B$} iff there exists a countable open cover $\{U_n\}_{n<\omega}$ such that $\overline{U_n}\cap B=\emptyset$. If there is no ambiguity in the context, we just say $A$ has the \textbf{covering property}.
\end{definition}
\begin{definition}
    Let $z=(x,y)\in \R^2$, $r>0$ in $\R$, denote the basic Sorgenfrey neighbourhood of radius $r$ around $z$ to be $U(z,r)= [x,x+r)\times[y,y+r)$.
\end{definition}
Since we will have two topologies on $\R$, we will use the following notation for closure and compliments.
\begin{definition}
    For $X\subseteq \R$ be a subset of reals, we denote $\overline{X}^E$ as the \textbf{Euclidean closure} and $\interior_E(X)$ as \textbf{Euclidean interior}. We denote $\overline{X}^{\sigma}$ and $\interior_\sigma(X)$ to denote the \textbf{Sorgenfrey closure} and \textbf{Sorgenfrey interior} respectively. 
\end{definition}
 The Sorgenfrey topology is finer than the Euclidean topology, i.e. every Euclidean open set is also Sorgenfrey open but not vice versa.

\begin{definition}
    Let $A,B\subseteq \R$, $f:A\to B$, denote the \textbf{graph} of $f$ in $\R^2$ to be $G_f=\{ (x,f(x)):x\in A\}$.
\end{definition}

All of our topological terminology is standard and we refer the reader to Engelking \cite{engelking1989general} for any undefined notions. 

\section{Q-sets and the Sorgenfrey square}

We begin by recalling a simple fact about monotone functions that we will use to establish a partial converse to Przymusi\'nski's theorem. 
\begin{lemma}\label{lem:homeo}
    Let $A,B$ be subsets of reals. For any monotone decreasing function $f:A\to B$, there is a countable set $B_0\subseteq B$ such that $f\upharpoonright (A\setminus f^{-1}(B_0)): A\setminus f^{-1}(B_0) \to f(A)\setminus B_0$ is a Euclidean homeomorphism.
    
    If, further, $f$ is strictly decreasing, then there is a countable $A_0\subseteq A$ such that $f\upharpoonright (A\setminus A_0):A\setminus A_0\rightarrow f(A\setminus A_0)$ is a Euclidean homeomorphism. \\
\end{lemma}

% \begin{trivlist}
% \item[\hskip \labelsep {\bfseries Lemma 2.1$'$}] \itshape
% \end{trivlist}

%{\color{red} I think this is not quite stated correctly The conclusion implies that the range of the original $f$ must be cocountable, but this is not necessarily true. I think you mean there is a countable $A_0$ such that $f\upharpoonright (A\setminus A_0):A\setminus A_0\rightarrow f(A\setminus A_0)$ is a homeomorphism. }

\begin{lemma}\label{lem:fsigma}
    Suppose $X,Y$ are sets of reals. Suppose that
    \begin{enumerate}
        \item there exists a strictly decreasing Euclidean homeomorphism $f:X \to Y$.
        \item For $A\subseteq X$, define
        \begin{align*}
        B&=X\setminus A\\
        G_{f|_{A}}&= \{(a,f(a)):a\in A\} \\
        G_{f|_{B}}&= \{(b,f(b):b\in B)\}.
    \end{align*}
    \end{enumerate}
    If there exists $\Uu,\Vv \subseteq X\times Y$ Sorgenfrey open such that $\Uu\supseteq G_{f|_A}, \Vv\supseteq G_{f_B}$ and $\Uu\cap\Vv=\emptyset$, then $A$ is relatively $F_\sigma$ in $X$ in the Euclidean topology.
\end{lemma}
\begin{proof}
     By assumption, let $\Uu,\Vv$ be Sorgenfrey open sets such that $\Uu\supseteq G_{f|_A}, \Vv\supseteq G_{f|_B}$.\\
    Note that $A=\bigcup_{m<\omega} A_m$ where $A_m=\{a\in A: U((a,f(a)),1/m)\subseteq \Uu\}$, we claim that 
    \begin{claim}\label{clm:closure}
        $\overline{A_m}^E \cap X\subseteq A$
    \end{claim}
    With this claim, we have $A=\bigcup_{m<\omega} (\overline{A_m}^E \cap X)$ which is a relatively Euclidean $F_\sigma$ set. So it suffices to prove the claim.
    \begin{proof}[Proof of \ref{clm:closure}]
        \;\\
        \begin{center}
            \includegraphics[]{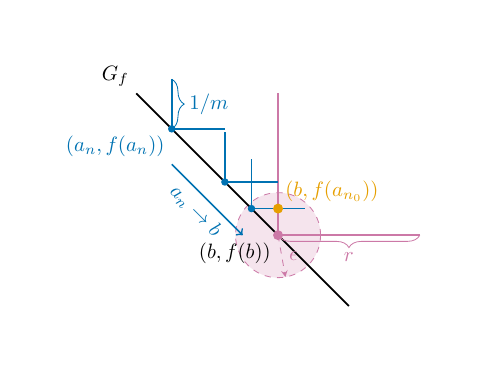}
        \end{center}
        Assume otherwise, if $a_n\to b$ for $a_n\in A_m$ and $b\in B$. There exists an $r>0$ such that $U((b,f(b)),r)\subseteq \Vv$. By continuity of $f$, we have $f(a_n)\to f(b)$ as well. Let $\epsilon <\min(1/m,r)$, without loss of generality there exists $a_{n_0}$ such that $b>a_{n_0}>b-\epsilon$ and $f(b)<f(a_{n_0})<f(b)+\epsilon$. Since $\epsilon<\min(1/m,r)$, we have
        \begin{align*}
            &f(b)<\mathbf{f(a_{n_0})}<f(b)+\epsilon< f(b) + r&&\implies \mathbf{f(a_{n_0})}\in[f(b), f(b)+r)\\
            &a_{n_0}<\mathbf{b}< a_{n_0}+\epsilon<a_{n_0}+1/m&&\implies \mathbf{b}\in [a_{n_0},a_{n_0}+1/m)
        \end{align*}
        Where the first expression implies
        $(b,f(a_{n_0}))\in U((b,f(b)),r))\subseteq \Vv$ However, the second expression implies $(a_{n_0},f(a_{n_0}))$ is also a point in $U((a_{n_0},f(a_{n_0})),1/m)\subseteq \Uu$, but $\Uu$ and $\Vv$ should have empty intersection hence this is a contradiction.
    \end{proof}
    
\end{proof}
We now prove the following partial converse to Przymusi\'nski's theorem:
\begin{proposition}\label{normalconverse}
     Let $(X[\leq])^2$ be normal. If there exists a strictly decreasing Euclidean homeomorphism $f:X\to X$, then $X$ is a $Q$-set.
\end{proposition}
\begin{proof}
    Since $f$ is an Euclidean homeomorphism, its graph $G_f$ is Euclidean closed and so Sorgenfrey closed. Since $f$ is strictly decreasing, $G_f$ is also Sorgenfrey discrete. Hence any subset of $G_f$ is Sorgenfrey closed. In particular, for any $A\subseteq X$, $G_{f|_A}$ and $G_{f|_{X\setminus A}}$ can be separated by Sorgenfrey open sets (by our normality assumption). By \ref{lem:fsigma}, $A$ is relatively Euclidean $F_\sigma$ in $X$. Hence $X$ is a $Q$-set.
\end{proof}
The following is a slightly more general partial result:
\begin{proposition}\label{Prop2.5}
    Let $(X[\leq])^2$ be normal. If there exists a strictly decreasing function $f:X\to X$ (not necessarily continuous with respect to the Euclidean topology), then $X\setminus C$ is a $Q$-set for some countable subset $C$.
\end{proposition}
\begin{proof}
    By \ref{lem:homeo}, there exists a countable set $C$ such that $f|_{X\setminus C}:X\setminus C\to f(X\setminus C)$ is a Euclidean homeomorphism. Then let $g=f|_{X\setminus C}$, the graph $G_g$ of $g$ is closed and discrete. Since $(X[\leq])^2$ is normal, any two subsets of $G_g$ can be separated by Sorgenfrey open subsets of $X^2$. In particular, let $X^\prime=X\setminus C$, take any subset $A\subseteq X^\prime$, $G_{g|_A}$ and $G_{g|_{X\prime\setminus A}}$ can be separated by open sets $\Uu,\Vv$. Take their intersection to $(X^\prime)^2$, $\Uu\cap (X^\prime)^2$ and $\Vv\cap(X^\prime)^2$ are still Sorgenfrey open in $(X^\prime)^2$. By \ref{lem:fsigma}, $A$ is relatively Euclidean $F_\sigma$ in $X^\prime$, hence $X^\prime$ is a $Q$ set as required.
    % Take any subset $A\subseteq X^\prime$, define the following sets
    % \begin{align*}
    %     B&=X^\prime\setminus A\\
    %     \Aa&= \{(a,f(a)):a\in A\}\\
    %     \Bb&= \{(b,f(b):b\in B)\}
    % \end{align*}
    % Note that the curly letters $\Aa,\Bb$ are the representatives of the set $A,B$ respectively on the graph of $g$. Since $G_g$ is close and discrete, any two subsets of $G_g$ is closed and can be separated by Sorgenfrey open sets. Let open sets $\Uu,\Vv$ be such that $\Uu\supseteq \Aa, \Vv \supseteq \Bb$.
    % Note that $A=\cup_{n<\omega} A_n$ where $A_n=\{a\in A: U((a,f(a)), 1/n)\subseteq \Uu\}$, we claim that 
    % \begin{claim}
    %     $\overline{A_n}^E\cap X^\prime \subseteq A$
    % \end{claim}
    % \begin{proof}
    %     Assume otherwise, if $a_n\to b$ for $a_n\in A$ and $b\in B$. There exists an $r>0$ such that $U((b,f(b)),r)\subseteq \Vv$. By continuity of $f$, we have $f(a_n)\to f(b)$ as well. Let $\epsilon = \min(1/n,r)$, Wlog there exists $a_{n_0}$ such that $b<a_{n_0}<b+\epsilon$ and $f(b)<f(a_{n_0})<f(b)+\epsilon$. Since $\epsilon<1/n$, $(a_{n_0},f(a_{n_0})+\epsilon)$ is a point in $U((a_{n_0},f(a_{n_)})),1/n)\subseteq \Uu$ However, we have 
    %     \[f(b)<f(a_{n_0})+\epsilon<f(b)+2\epsilon<f(b)+r\]
    %     So $(a_{n_0},f(a_{n_0}))$ is also a point in $U((b,f(b)),r)\subseteq \Vv$, but $\Uu$ and $\Vv$ should have empty intersection hence this is a contradiction.
    % \end{proof}
    % Now with this claim, we have $A=\cup_{n<\omega} (\overline{A_n}^E \cap X^\prime)$ which is a relatively $F_\sigma$ set.
\end{proof}
Baumgartner proved that PFA implies that any two $\aleph_1$-dense sets of reals are order isomorphic \cite{Baumgartner1973}. This consequence of PFA has come to be referred to as Baumgartner's Axiom (BA). From the proof of the above proposition we obtain the following corollary:
\begin{cor}\label{BA}
    BA implies that if $X$ is an uncountable set of reals and $(X[\leq])^2$ normal then $X$ contains an uncountable $Q$-set. 
\end{cor}
\begin{proof} Any uncountable $X$ contains an $\aleph_1$-dense subset $X^\prime$. By BA it is order isomorphic to its reverse order. So by the proof of Proposition \ref{Prop2.5} there is a countable $C$ such that $X^\prime\setminus C$ must be a $Q$-set.\end{proof}

For the remainder of this section we focus on constructing, from CH, an uncountable set of reals whose square, under the Sorgenfrey topology, is normal. Of course, under CH, no uncountable set of reals is a Q-set and so the converse of Przymusi\'nski's theorem fails assuming CH. 

We begin by recalling another simple folklore result on monotone real-valued functions. For completeness sake, we provide a proof. 

\begin{lemma}\label{lem:extend}
    Let $f:A\to \R$ be strictly decreasing on uncountable $A\subseteq \R$, then there is a strictly decreasing homeomorphism $F$ on $\overline{A}^E \setminus C$ for some countable set $C$ such that $f(a)=F(a)$ for all $a\in A\setminus C$.
\end{lemma}
\begin{proof}
    $f$ can be extended to a decreasing function on $\overline{A}^E$. Note that this function is not necessarily strictly decreasing. By \ref{lem:homeo}, we can remove a countable set $B_0$ such that  $f$ is a Euclidean homeomorphism $F$ on $\overline{A}^E\setminus f^{-1}(B_0)$ hence strictly decreasing. We will show $f^{-1}(B_0)$ is countable by showing for every $b\in B_0$, $f^{-1}(b)$ has at most 3 points. Assume there are 4 points $x<y<z<w$ in $f^{-1}(b)$. Note that between $x$ and $w$ there is at most 1 point in $A$, since for every $z_1<z_2\in A$,
    \[f(x)\geq f(z_1)\geq f(z_2)\geq f(w)=f(x)\]
    but $f$ is strictly decreasing on $A$, so there can only be at most 1 point in $A$. Hence, there are no points in $\overline{A}^E\setminus A$. But $y,z$ can not be both in $A$, we have a contradiction. Let $C=f^{-1}(B_0)$, we have the result.
\end{proof}

 We also recall definition of entangled sets from \cite{Todorcevic1985}. Let $L\subseteq \R$, denote $(L)^n\subseteq L^n$ as the set of all $\langle x_0,\dots, x_{n-1}\rangle\in L^n$, such that $x_0<x_1<\dots<x_n$. A set $L\subseteq \R$ is $\mathbf{n}$\textbf{-entangled} iff for any set of disjoint $n$-tuples $A\subseteq (L)^n$ of size $\omega_1$ and for every $s\in 2^n$ there exists $x,y\in A$ such that $\forall i\leq n, (x_i<y_i)\iff s_i=0$. $L$ is \textbf{entangled} iff $L$ is $\mathbf{n}$\textbf{-entangled} for all $n<\omega$. Let $s\in 2^n$, $A\subseteq (L)^n$ is said to have \textbf{type} $s$ iff  $\forall i\leq n(x_i<y_i)\iff s_i=0$.\\

The following lemma is proven in \cite{Todorcevic1985} (see property (5) in the proof of  Theorem 1 in \cite{Todorcevic1985}).

\begin{lemma}\cite{Todorcevic1985}\label{lem:entangled}
    Let $E\subseteq \R$ be such that for any continuous function $f$ from a $G_\delta$ subset of ${\mathbb R}^n$ into ${\mathbb R}$, $(n<\omega)$ there is an $\alpha<2^\omega$ such that
        \[
        \forall \beta\geq \alpha, f(\{r_\gamma:\gamma<\beta\}^n) \cap E \subseteq \{r_\gamma:\gamma<\beta\}
        \]

    Then $E$ is entangled.
\end{lemma}

\begin{theorem}\label{setE}
    Under CH, there exists a set $E$ such that 
    \begin{enumerate}
        \item For any continuous function $f$ from a $G_\delta$ subset of ${\mathbb R}^n$ into $E$, $(n<\omega)$ there is an $\alpha<2^\omega$ such that
        \[
        \forall \beta\geq \alpha, f(\{r_\gamma:\gamma<\beta\}^n) \cap E \subseteq \{r_\gamma:\gamma<\beta\}
        \]
        \item $E$ contains $\Q$ and is concentrated on ${\mathbb Q}$. 
    \end{enumerate}
\end{theorem}
Recall that $E$ is concentrated on $\Q$ means that $E\setminus U$ is countable for every open set $U\supseteq \Q$, so $E$ is not a $\lambda$-set since $\Q$ is not a relative $G_\delta$-set in $E$.
\begin{proof}
    Define
    \begin{align*}
        \Ff=\{&f:U\to \R:  U \subseteq \R^n,n<\omega, U \text{ is } G_\delta\}\}\\
        \Aa= \{&A\subseteq\R: A \text{ is open},\ A\supseteq \Q\}
    \end{align*}
    Note that both $\Ff,\Aa$ have cardinality $\mathfrak{c}$, so under $CH$ we can enumerate them in $\omega_1$ as follows
    \[\Ff=\{f_{\alpha}:\alpha<\omega_1\}, \Aa=\{A_\alpha:\alpha<\omega_1\}. \]
     For every $\alpha<\omega_1$ let $n_\alpha$ be the parity of the function $f_\alpha$ and define the following set:
    \begin{align*}
       % B_\alpha&=\{x\in \R: f_\alpha(x)\neq x\}\\
        C_\alpha&= \bigcap_{\beta\leq \alpha} A_\beta \setminus \Q.
    \end{align*}
    Note that $C_\alpha$ is uncountable by the Baire Category Theorem. We now define elements $x_\alpha\in \R$ and countable sets $I_\alpha\subseteq {\mathbb R}$ recursively as follows: \\
    \textbf{Stage 0:} Let $I_0=f_0({\mathbb Q}^{n_0})$ and take $x_0\in C_0\setminus I_0$.\\
    \textbf{Stage $\boldsymbol{\alpha<\omega_1}$:} We have $x_\beta$ and $I_\beta$ defined already for all $\beta<\alpha$. For any $\beta\leq \omega_1$, let $n_\beta$ be the parity of $f_\beta$, and let
    \begin{align*}
        X_\alpha&= {\mathbb Q}\cup \{x_\beta:\beta<\alpha\} \\
     %   Y_\alpha&= \{y_\beta:\beta<\alpha\}\\
        I_\alpha&= \bigcup_{\beta\leq \alpha} f_\beta(X_\alpha^{n_\beta})\\
    \end{align*}
    Note that all of the above sets are countable, and $C_\alpha$ is uncountable. Hence, we can choose $x_\alpha\in C_\alpha\setminus (X_\alpha\cup I_\alpha )$. \\
    Finally we define $E=\Q\cup \{x_\alpha:\alpha<\omega_1\}$.\\
     \textbf{Verify Property (1):} Fix $n<\omega$, for any continuous function $g:A\to E$, by choice of $x_\alpha$, we have that for any $\beta \geq \alpha$,
     $f_{\alpha_0}(\{x_\gamma:\gamma <\alpha\}^n)$ can not take on any values in $\{x_\beta:\beta\geq\alpha_0\}$. Indeed for all $\beta>\alpha_0,$ $x_\beta\not \in I_\alpha \supseteq f_{\alpha_0}((\{x_\alpha:\alpha<\alpha_0\}\cup \Q)^n)$\\
    \textbf{Verify Property (2):} That $E$ is concentrated on $\Q$ follows easily from the fact that for every for every open set containing ${\mathbb Q}$ is enumerated as $A_\alpha$ and so $C_\beta\subseteq A_\alpha$ for all $\beta\geq \alpha$. \\
\end{proof}
As mentioned above, the set $E$ constructed is not a $\lambda$-set so also not a $Q$-set.

%\begin{lemma}\label{lem:ctbpreim}
 %   Let $X$ be hereditarily Lindelof, and $p_0:X^2\to X$ be the projection map on the first coordinate, let $A\subseteq X$ be a countable set, then $p_0^{-1}(A)$ is Lindelof. (Same is true for $p_1$, the projection map on the second coordinate).
%\end{lemma}
%\begin{proof}
%  Note $p_0^{-1}(A) =\cup_{a\in A} p_0^{-1}(\{a\})$ a countable union. So it suffices to prove $p_0^{-1}(\{a\})$ is Lindelof. Indeed, there is a homeomorphism from $p_0^{-1}(\{a\})$ to $X$, and $X$ is hereditarily Lindelof, hence $p_0^{-1}(\{a\})$ is Lindelof. Hence $p_0^{-1}(A)$ is Lindelof.
%\end{proof}
By \ref{lem:entangled}, the set $E$ is also entangled and in particular $2$-entangled. By \ref{lem:extend} this implies no monotone functions can be defined on an uncountable subset of $E$.

For our next observation, we will use the following result, proven by Przymusi\'nski (see the proof of Theorem 1 in \cite{T.Przymusinski1973}). 
\begin{lemma}\label{lem:przymusinski}
    Let $X\subseteq \R$ be any subset of reals. Fix any Sorgenfrey closed set $A,B \subseteq X^2$, and a Sorgenfrey open cover $\Aa$ of $A$. For any $z\in B^c$ choose $n_z\in \N$ such that $U(z,1/n_z)\cap B=\emptyset$, this is possible since $B$ is Sorgenfrey closed. Let $A\cap \overline{B}^E=\bigcup_{q\in\Q}A_q$ where
    \[A_q=\{ z=(x,y)\in A_1:x+y<q<x+y+1/n_z \}\]
    For any $q\in \Q$, we have
    \begin{enumerate}[label=(\arabic*)]
        \item $A_q= G_q \cup L_q$, where $L_q$ can be covered by countably many Sorgenfrey open sets of $\Aa$ and $G_q$ is a graph of some strictly decreasing function, 
        \item If every $G_q$ can be covered by countably many Sorgenfrey open sets in $\Aa$, then $X[\leq]^2$ is normal.
    \end{enumerate}
\end{lemma}
Recall, Todor\v cevi\'c noted \cite{Todorcevic1989} that if $E$ is entangled, then $E[\leq]$ is hereditarily Lindel\"of in all finite powers. In relation to this we have: 
\begin{theorem}\label{thm:E^2normal}
    Let $E$ be any subset of ${\mathbb R}$ that admits no uncountable strictly decreasing functions (e.g., the set constructed in Theorem \ref{setE}), then $(E[\leq])^2$ is normal.
\end{theorem}
\begin{proof}
    Following the notation of \ref{lem:przymusinski} (1): $G_q$ is the graph of a strictly decreasing function, but by hypothesis on $E$, each $G_q$ is countable, hence Lindel\"of. By \ref{lem:przymusinski} (2), $(E[\leq])^2$ is normal.  
\end{proof}
Now one may ask if every $E\subseteq \R$ such that $(E[\leq])^2$ is normal is either a $Q$-set or is $2$-entangled. However, the following example shows that is not the case.

\begin{theorem}
    Assuming CH, there exists a set $E$ such that 
    \begin{enumerate}
        \item $E$ is not 2-entangled.
        \item $E$ is concentrated on $\Q$ (hence not a Q-set nor a $\lambda$-set).
        \item $(E[\leq])^2$ is normal.
    \end{enumerate}
    \begin{proof}
        We will construct our set $E$ by preserving the strictly increasing function $x\mapsto x+1$ while at the same time making it concentrated on ${\mathbb Q}$ and killing all the strictly decreasing functions on uncountable sets. As in \ref{setE}, let $\Ff=\{f_\alpha:\alpha<\omega_1\}$ enumerate all strictly decreasing functions defined on Borel subsets of the reals, $\Aa=\{ A_\alpha:\alpha<\omega_1\}$ an enumeration of all Euclidean open sets containing $\Q$ and, for each $\alpha<\omega_1$, let $C_\alpha=\cap_{\beta\leq\alpha} A_\alpha \setminus \Q$. We now define elements $x_\alpha\in\R$ recursively as follows:\\
        \textbf{Stage 0:} Take $x_0\in C_0$ such that $f_0(x_0)\not=x_0$ and $x_0 \not \in f_0(\Q)\cup f_0^{-1}(\Q)$\\
    \textbf{Stage $\boldsymbol{\alpha<\omega_1}$:} We have $x_\beta$ defined already for all $\beta<\alpha$, define
    \begin{align*}
        X_\alpha&= \{x_\beta:\beta<\alpha\} \cup \{x_\beta+1: \beta<\alpha\} \cup \Q\\
     %   Y_\alpha&= \{y_\beta:\beta<\alpha\}\\
        I_\alpha&=\{ f_\beta(x): \beta\leq \alpha, x\in X_\alpha \}\\
        P_\alpha&=\{ f_\beta^{-1}(x): \beta\leq\alpha, x\in X_\alpha\}\\
    \end{align*}
    Note that all of the above sets are countable, and $C_\alpha$ is uncountable. Hence, we can choose $x_\alpha\in C_\alpha\setminus (X_\alpha\cup I_\alpha\cup P_\alpha)$.    \\
    Finally we define $E=\Q\cup \{x_\alpha:\alpha<\omega_1\}\cup\{x_\alpha+1:\alpha<\omega_1\}$.\\
    \textbf{Verify Property (1): }$E$ is not 2-entangled since the uncountable one-to-one increasing function $x\mapsto x+1$ exists.\\ 
    \textbf{Verify Property (2): }This follows exactly as in \ref{setE}.\\
    \textbf{Verify Property (3): }We will show $(E[\leq])^2$ is normal using \ref{thm:E^2normal}. 
    Fix any strictly decreasing function $g$ defined on subsets of $E$. By \ref{lem:extend}, there exists $F$ a strictly decreasing homeomorphism defined from a subset of $E$ into $E$ such that $dom(F)\setminus dom(g)$ is countable. There exists an $\alpha_0<\omega_1$ such that $F=f_{\alpha_0}$ on $dom(F)$. We will show the $f_{\alpha_0}$ can only be defined on a countable set, hence $dom(g)$ is countable and $E$ admits no strictly decreasing function on uncountable subsets. Indeed, for all $\beta>\alpha_0$, $f_{\alpha_{0}}(x_{\beta})$ can not be mapped to $\{x_{\beta}:\beta \geq \alpha_0\}\}$. For all $\gamma\geq \alpha_0$,
    \begin{align*}
        \text{If } \alpha_0 <\gamma <\beta,& \text{ we have that when choosing }x_\beta, \text{ we made sure}\\
        &x_\beta\not \in P_{\beta}\ni f_{\alpha_0}^{-1}(x_\gamma)\implies f_{\alpha_0}(x_{\beta})\neq x_\gamma  \\ 
        \text{If } \alpha_0 <\gamma <\beta,& \text{ we have that when choosing } x_\gamma, \text{ we made sure} \\
        &x_\gamma \not \in P_{\gamma}\ni f(x_{\beta})\implies x_\gamma\neq f_{\alpha_0}(x_\beta) 
    \end{align*}
    For the same reason, $f_{\alpha_0}(x_\beta)$ cannot be mapped to $\{ x_{\beta}+1: \beta\geq a_0\}$. Since $f_{\alpha_0}$ is one-to-one, $f_{\alpha_0}$ has countable domain.
    \end{proof}
\end{theorem}

\section{$\lambda$-sets and the Sorgenfrey square}

\begin{theorem}
    If $X$ is a $\lambda$-set, then $(X[\leq])^2$ is pseudo-normal.
\end{theorem}
\begin{proof}
    Let $A,B$ be closed sets and assume $B$ is countable. 
    Recall the shoelace lemma: to separate $A$ and $B$ it suffices to cover each set by a countable family of open sets whose closures are disjoint from the other closed set.
    
    It suffices to show that $A$ has the covering property with respect to $B$. Since $B$ is countable it clearly has the covering property with respect to $A$.
 %   \begin{claim}
  %      $B$ has the covering property w.r.t. $A$.
   % \end{claim}
    %\begin{proof}
     %   For any $b\in B$, there exists $r_b>0$ such that  $U(b,r_b)\cap A=\emptyset$. Since $U(b,r_b)$ is also closed, we have $\overline{U(b,r_b)}\cap A=\emptyset$. Then since $B$ is countable, the cover $\{U(b,r_b)\}_{b\in B}$ is the required countable cover.
  %  \end{proof}
    We begin by defining the following sets:
    \begin{align*}
        B_0&= p_0(B)\\
        B_1&=p_1(B)\\
        A_0&= A\cap p_0^{-1}(B_0)\\
        A_1&= A\cap p_1^{-1}(B_1)\\
        W&= A\setminus(A_0\cup A_1)
    \end{align*}
    Note that $B_0$ is countable, then $p_0^{-1}(B_0)= \bigcup_{b\in B_0} p_0^{-1}(b) $ is a countable union of Lindel\"of spaces which is still Lindel\"of. Since $A$ is Sorgenfrey closed, $A_0$ is a closed subspace of a Lindel\"of space, which is still Lindel\"of. By the similar reason, $A_1$ is Lindel\"of as well. Hence both $A_0$ and $A_1$ have the covering property. Now it suffices to prove 
    \begin{claim}
        $W$ has the covering property
    \end{claim}
    \begin{proof}[Proof of Claim]
        We know that there are families $\{F_n:n\in \omega\}$ and $\{H_n:n\in \omega\}$ of relatively closed (in the Euclidean topology) subsets of $X$ such that 
        \begin{align*}
            p_0(W)\subseteq p_0(B)^c &= \cup_{n\in \omega} F_n\\
            p_1(W)\subseteq p_1(B)^c &= \cup_{m\in\omega} H_m
        \end{align*}
        For any $z\in W$, there exists $s_z>0$ such that $U(z,s_z)\cap B=\emptyset$. 
        Define $$W_{n,m}=\bigcup\{U(z,s_z):\ z=(x,y)\in W, x\in F_n,y\in H_m\}.$$
        We see $W\subseteq\cup_{n,m} W_{n,m}$ and each $W_{n,m}$ is Sorgenfrey open. We know that $W_{n,m}\cap B=\emptyset$, so it suffices to show $(\overline{W_{n,m}}^\sigma\setminus W_{n,m})\cap B=\emptyset$. 
        
        Assume otherwise. Then there exists a $z\in B\setminus W_{n,m}$ such that for all $k\in \omega$, there exists a $w_k=(x_k,y_k)\in W_{n,m}$ such that, $U(z,1/k)\cap U(w_k,s_{w_k})\neq \emptyset$.  We have the following three cases depicted below.\\
        \includegraphics[scale=0.88]{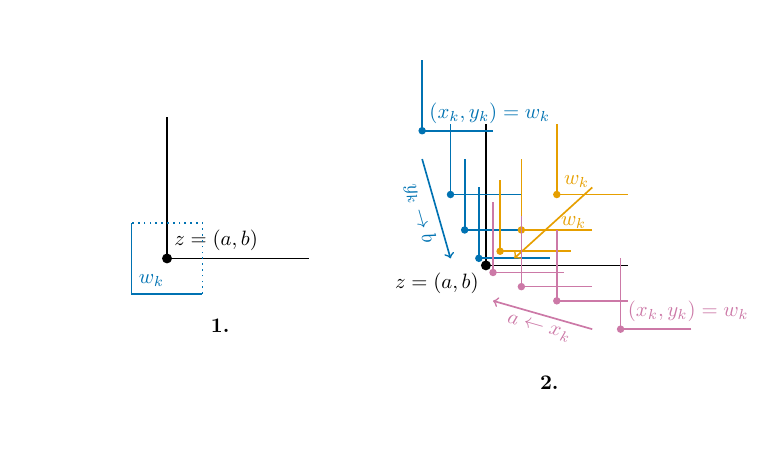}
        \begin{enumerate}
            \item For infinitely many $k$ we have $x_k<a,y_k<b$\\
            In this case, we have $(a,b)\in U(w_k,s_{w_k})\subseteq W_{n,m}$, which is a contradiction since we chose $z=(a,b)$ to be disjoint from $W_{n,m}$
            \item For infinitely many $k$ $x_k\to a$. But since each $x_k$ is in $F_n$, it follows that $a\in F_n$ which is impossible.
            \item For infinitely many $k$ $y_k\to b$. But since $y_k$ is in $H_m$ it would follow that $b\in H_m$, also impossible. 
        \end{enumerate}
        Hence $\overline{W_{n,m}}\cap B=\emptyset$, we have the desired covering. 
    \end{proof}
\end{proof}

From the previous example, we know that $(X[\leq])^2$ may be normal, while $X$ is not a $\lambda$-set. However, we do have the following partial converse analogous to Proposition \ref{normalconverse}.
\begin{proposition}
    Let $(X[\leq])^2$ be pseudo-normal. If there exists a strictly decreasing homeomorphism $f:X\to X$ with respect to the Euclidean topology, then $X$ is a $\lambda$-set. 
\end{proposition}
\begin{proof}
    Since $f$ is strictly decreasing we have that the graph $G_f$ is closed and discrete in $(X[\leq])^2$. Therefore any subset of $G_f$ is closed. Let $A$ be any countable subset, we have $G_{f|_A}$ is countable as well, hence $G_{f|_A}$ and $G_{f|_{X\setminus A}}$ are closed and can be separated by Sorgenfrey open sets. Then by \ref{lem:fsigma}, $X\setminus A$ is relatively $F_\sigma$ in $X$, hence $X$ is a $ \lambda$-set.
\end{proof}

Finally, we close with some remarks and a question. As we have seen, the converse to Przymusi\'nski's theorem is independent (e.g., assuming CH). Is the converse consistent? I.e., is it consistent that for an uncountable subset of the reals $X$, if $(X[\leq])^2$ is normal then $X$ must be a Q-set. This would vacuously follow from a positive answer to the following possibly more interesting question:

\begin{que} Is it relatively consistent with ZFC that $(X[\leq])^2$ is never normal for $X$ an uncountable set of reals?
\end{que}

By Corollary \ref{BA} if Baumgartner's Axiom holds and there are no $Q$-sets, then $X([\leq])^2$ is never normal for any uncountable set of reals $X$. So we ask, 
\begin{que} Does Baumgartner's Axiom imply that there is a $Q$-set?
\end{que}

\bibliographystyle{abbrv}
\bibliography{references}

\end{document}